\documentclass[11pt,headings=small]{scrartcl}

\usepackage[sumlimits]{amsmath}
\usepackage{amssymb,amsfonts,amsthm}
\usepackage[english]{babel}
\usepackage[latin1]{inputenc}
\usepackage[T1]{fontenc}
\usepackage{lmodern}
\usepackage{graphicx}
\usepackage{scrpage2}
\usepackage{hyperref}
\usepackage{color}
\usepackage{setspace}
\usepackage{textcomp}
\usepackage{enumerate}
\usepackage{booktabs}

\newtheorem{thm}{Theorem}[section]
\newtheorem{prp}[thm]{Proposition}
\newtheorem{cor}[thm]{Corollary}
\newtheorem{lem}[thm]{Lemma}
\theoremstyle{definition}
\newtheorem{df}{Definition}[section]
\newtheorem*{rem}{Remark}

\def\P{\mathcal{P}}
\def\S{\mathcal{S}}
\def\R{\mathbb{R}}

\def\N{\mathbb{N}}
\def\Q{\mathcal{Q}}

\def\F{\mathbb{F}}

\def\D{\mathbb{D}}

\def\1{\mathbbm{1}}

\newcommand{\bsj}{\boldsymbol{j}}
\newcommand{\bsk}{\boldsymbol{k}}
\newcommand{\bsm}{\boldsymbol{m}}
\newcommand{\bst}{\boldsymbol{t}}
\newcommand{\bsx}{\boldsymbol{x}}
\newcommand{\bsy}{\boldsymbol{y}}
\newcommand{\bsz}{\boldsymbol{z}}
\newcommand{\bszero}{\boldsymbol{0}}
\newcommand{\bsone}{\boldsymbol{1}}
\newcommand\dint{\,{\rm d}}

\newcommand{\ld}{{\rm ld}\,}

\DeclareMathOperator{\bmo}{BMO}

\allowdisplaybreaks[1]
\clearscrheadfoot
\ihead{\headmark}
\ohead{\thepage}

\begin{document}
\pagestyle{scrheadings}
\onehalfspacing

\title{Optimal $L_p$-discrepancy bounds for second order digital sequences}
\author{Josef Dick\thanks{This research was supported under Australian Research Council's Discovery Projects funding scheme (project number DP150101770).}, Aicke Hinrichs, Lev Markhasin, Friedrich Pillichshammer\thanks{F.P. is supported by the Austrian Science Fund (FWF): Project F5509-N26, which is a part of the Special Research Program "Quasi-Monte Carlo Methods: Theory and Applications".}}
\maketitle

\begin{abstract}
The $L_p$-discrepancy is a quantitative measure for the irregularity of distribution modulo one of infinite sequences. In 1986 Proinov proved for all $p>1$ a lower bound for the $L_p$-discrepancy of general infinite sequences in the $d$-dimensional unit cube, but it remained an open question whether this lower bound is best possible in the order of magnitude until recently. In 2014 Dick and Pillichshammer gave a first construction of an infinite sequence whose order of $L_2$-discrepancy matches the lower bound of Proinov. Here we give a complete solution to this problem for all finite $p > 1$. We consider so-called order $2$ digital $(t,d)$-sequences over the finite field with two elements and show that such sequences achieve the optimal order of $L_p$-discrepancy simultaneously for all $p \in (1,\infty)$.
\end{abstract}

\centerline{\begin{minipage}[hc]{130mm}{
{\em Keywords:} $L_p$-discrepancy, higher order digital sequences, quasi-Monte Carlo\\
{\em MSC 2010:} 11K06, 11K38}
\end{minipage}} 

\section{Introduction}
Let $d,N \in \mathbb{N}$ (where $\mathbb{N}=\{1,2,3,\ldots\}$) and let $\P_{N,d}$ be an $N$-element point set in the unit cube $[0,1)^d$. The discrepancy function of $\P_{N,d}$ is defined as
\begin{align}\label{deflocdisc}
D_{\P_{N,d} }(\bsx) = \frac{1}{N}\sum_{\bsz \in \P_{N,d}} \chi_{[\bszero,\bsx)}(\bsz) - x_1 \cdots x_d
\end{align}
where $\bsx = (x_1, \ldots, x_d) \in [0,1]^d$ and $[\bszero,\bsx)=[0,x_1)\times\ldots\times[0,x_d)$. By $\chi_A$ we mean the characteristic function of a set $A\in\R^d$, i.e., $\chi_A(\bsx)=1$ if $\bsx \in A$ and 0 if $\bsx \not\in A$. The term $\sum_{\bsz} \chi_{[\bszero,\bsx)}(\bsz)$ in \eqref{deflocdisc} is equal to the number of points of $\P_{N,d}$ in $[\bszero,\bsx)$. Hence, $D_{\P_{N,d}}$ is a normalized measure for the deviation of the proportion of the number of points of $\P_{N,d}$ in $[\bszero, \bsx)$ from the `fair' or `expected' proportion of the number of points $ \lambda_d([\bszero,\bsx)) = x_1 \cdots x_d$ in this interval under the assumption of a perfect uniform distribution. Here $\lambda_d$ denotes the $d$-dimensional Lebesgue measure.

The $L_p$-discrepancy of $\P_{N,d}$ is defined as the $L_p$-norm of the discrepancy function, i.e., 
\[ L_{p,N}(\P_{N,d}) = \|D_{\P_{N,d}}|L_p([0,1]^d)\| \ \ \ \ \ \mbox{ for $p \in [1,\infty]$.} \] 

For an infinite sequence $\S_d$ in $[0,1)^d$ and $N \in \mathbb{N}$ the discrepancy function is defined as $D_{\S_d}^N(\bsx)=D_{\P_{N,d}}(\bsx)$, where the point set $\P_{N,d}$ consists of the first $N$ terms of $\S_d$, and the $L_p$-discrepancy of $\S_d$ is defined as
\[ L_{p,N}(\S_d) = \|D_{\S_d}^N|L_p([0,1]^d)\| \ \ \ \ \ \mbox{ for $p \in [1,\infty]$.} \]

The $L_p$-discrepancy is a quantitative measure for the irregularity of distribution of finite point sets and of infinite sequences. We refer to \cite{DT,KN} for extensive introductions to this topic. It is well known that a sequence $\S_d$ is uniformly distributed modulo one in the sense of Weyl \cite{W} if and only if $L_{p,N}(\S_d)$ tends to zero for $N \rightarrow \infty$. The $L_p$-discrepancy is also closely related to the worst-case integration error in certain function spaces using quasi-Monte Carlo algorithms via variants of the Koksma-Hlawka inequality. This follows immediately from Hlawka's identity (which is also sometimes attributed to Zaremba); see \cite{hl61,zar} or also \cite{DP10,LP14,NW08}.

The conceptual difference between the discrepancy of finite point sets and infinite sequences can be explained in the following way (cf. \cite{Mat99}): while for finite point sets we are interested in the behavior of the whole set $\{\bsx_0,\bsx_1,\ldots,\bsx_{N-1}\}$ with a fixed number of elements $N$, for infinite sequences we are interested in the discrepancy of all initial segments $\{\bsx_0\}$, $\{\bsx_0,\bsx_1\}$, $\{\bsx_0,\bsx_1,\bsx_2\}$, \ldots, $\{\bsx_0,\bsx_1,\bsx_2,\ldots,\bsx_{N-1}\}$, where $N=2,3,4,\ldots$. In this sense the discrepancy of finite point sets can be viewed as a static setting and the discrepancy of infinite sequences as a dynamic setting. Very often the dynamic setting in dimension $d$ is related to the static setting in dimension $d+1$ (see, for example, \cite[Chapter~2.2, Theorem~2.2, Example~2.2]{KN}). This will also be confirmed by our results.\\

It is well known that for every $p \in (1,\infty]$ and every $d \in \mathbb{N}$ there exists a positive constant $c_{p,d}$ with the following property:  for every finite $N$-element point set $\P_{N,d}$ in $[0,1)^d$ with $N \ge 2$ we have 
\begin{equation}\label{lbdlpdipts}
L_{p,N}(\P_{N,d}) \ge c_{p,d} \frac{(\log N)^{\frac{d-1}{2}}}{N}.
\end{equation}
This has been first shown in a celebrated paper by Roth~\cite{R54} for $p \ge 2$ and by Schmidt~\cite{S77} for $p \in (1,2)$. As shown by Hal\'{a}sz~\cite{H81} the estimate is also true for $p=1$ and $d=2$, i.e.,  there exists a positive constant $c_{1,2}$ with the following property:  for every finite $N$-element point set $\P_{N,2}$ in $[0,1)^2$ with $N \ge 2$ we have 
\begin{equation}\label{lbdl1D2dipts}
L_{1,N}(\P_{N,2}) \ge c_{1,2} \frac{(\log N)^{\frac{1}{2}}}{N}.
\end{equation} 

Later Proinov \cite{P86} (see also \cite{DP14b} for a proof) extended these results to infinite sequences: for every $p \in(1,\infty]$ and every $d \in \mathbb{N}$, $d \ge 2$, there exists a positive constant $c_{p,d}$ with the following property: for every infinite sequence $\S_d$ in $[0,1)^d$ we have 
\begin{equation}\label{lbdlpdiseq}
L_{p,N}(\S_d) \ge c_{p,d} \frac{(\log N)^{\frac{d}{2}}}{N} \ \ \ \ \mbox{ for infinitely many $N \in \mathbb{N}$.}
\end{equation}
For $d=1$ this estimate is also valid for $p=1$, i.e., there exists a positive constant $c_{1,1}$ with the following property:  for every infinite sequence $\S_1$ in $[0,1)$ we have 
\begin{equation}\label{lbdl1D1diseq}
L_{1,N}(\S_1) \ge c_{1,1} \frac{(\log N)^{\frac{1}{2}}}{N}.
\end{equation}
This can be shown by combining Proinov's method \cite{P86} (see also \cite{DP14b}) with the result of Hal\'{a}sz \eqref{lbdl1D2dipts}.

The lower bound \eqref{lbdlpdipts} for finite point sets is known to be best possible in the order of magnitude in $N$, i.e., for every $d,N \in \mathbb{N}$, $N \ge 2$, one can find an $N$-element point set $\P_{N,d}$ in $[0,1)^d$ with $L_p$-discrepancy of order 
\begin{equation}\label{uplpps}
L_{p,N}(\P_{N,d}) \ll_{p,d} \frac{(\log N)^{\frac{d-1}{2}}}{N}.
\end{equation}
For functions $f,g:D \subseteq \mathbb{N} \rightarrow \mathbb{R}$ with $g \ge 0$ we write $f(N) \ll g(N)$ if there exists some $C>0$ such that $f(N) \le C g(N)$ for all $N \in D$. If we want to stress that $C$ depends on some parameters, say $a,b$, then we indicate this by writing $f(N) \ll_{a,b} g(N)$. If we have $f(N) \ll_{a,b} g(N)$ and $g(N) \ll_{a,b} f(N)$ then we write $f(N) \asymp_{a,b} g(N)$.

The result in \eqref{uplpps} was proved by Davenport \cite{D56} for $p= 2$, $d= 2$, by Roth \cite{R80} for $p= 2$ and arbitrary $d$ and finally by Chen \cite{C80} in the general case. Other proofs were found by Frolov~\cite{Frolov}, Chen~\cite{C83}, Dobrovol'ski{\u\i}~\cite{Do84}, Skriganov~\cite{Skr89, Skr94}, Hickernell and Yue~\cite{HY00}, and Dick and Pillichshammer~\cite{DP05b}. For more details on the history of the subject see the monograph \cite{BC}. Apart from Davenport, who gave an explicit construction in dimension $d=2$, these results are pure existence results and explicit constructions of point sets where not known until the beginning of this millennium. First explicit constructions of point sets with optimal order of $L_2$-discrepancy have been provided in 2002 by Chen and Skriganov \cite{CS02} for $p=2$ and in 2006 by Skriganov \cite{S06} for general $p$. Other explicit constructions are due to Dick and Pillichshammer \cite{DP14a} for $p=2$, and Dick \cite{D14} and Markhasin \cite{M15} for 
general $p$.

It is also known that the lower bound \eqref{lbdlpdiseq} for infinite sequences is best possible for the particular case $p \in (1,2]$ (and also for $p=1$ when $d=1$). This was first shown by Dick and Pillichshammer~\cite{DP14a} who gave an explicit construction of sequences whose $L_2$-discrepancy achieves the order of magnitude $(\log N)^{d/2}/N$ for all $N \ge 2$. For dimension $d=1$ there exists a simple construction of a sequence with optimal order of $L_p$-discrepancy for all $p \in [1,\infty)$. Let $\mathcal{V}=(y_n)_{n \ge 0}$ be the van der Corput sequence (in base $2$), i.e., $y_n=\sum_{j\ge 0} \tfrac{n_j}{2^{j+1}}$ whenever $n \in \mathbb{N}_0$ has binary expansion $n=\sum_{j \ge 0} n_j 2^j$ with digits $n_j \in \{0,1\}$ (which is of course finite). Then let $\mathcal{V}^{{\rm sym}}=(z_n)_{n \ge 0}$ be the so-called symmetrized van der Corput sequence given by $z_{2n}=y_n$ and $z_{2n+1}=1-y_n$ for $n \in \mathbb{N}_0$. Then it has been shown in \cite{KP15} that for all $p \in [1,\infty)$ we have 
$$L_{p,N}
 (\mathcal{V}^{{\rm sym}}) \ll_p \frac{(\log N)^{\frac{1}{2}}}{N}\ \ \ \ \mbox{ for all $N\ge 2$.}$$ A generalization of this result to van der Corput sequences in arbitrary base $b \ge 2$ has been shown quite recently by Kritzinger \cite{K15}. See also the recent survey article \cite{fkp15} and the references therein for more information about symmetrized van der Corput sequences.\\

In this paper we provide explicit constructions of infinite sequences in arbitrary dimensions $d$ whose $L_p$-discrepancy is of order of magnitude $(\log N)^{d/2}/N$ for all $p \in [1, \infty)$. Thereby we prove that the lower estimate \eqref{lbdlpdiseq} is best possible for all $p \in (1,\infty)$.

The following is the main result of this work.
\begin{thm} \label{main_result}
There exists an explicit construction of an infinite sequence $\S_d$ in $[0,1)^d$ with the property that  
\begin{align*}
 L_{p,N}(\S_d) \ll_{p,d} \frac{\left(\log N\right)^{\frac{d}{2}}}{N}\ \ \ \ \mbox{ for all $N\ge 2$ and all $1 \le p < \infty$.}
\end{align*}
\end{thm}

A more concrete version of the main result will be stated in Section~\ref{sec_exp_constr} as Theorem~\ref{thm_main2} and proved in Section~\ref{sec_proof_main}. For $p=2$ this result has been shown in \cite{DP14a} (but with a more complex construction - see the discussion after Theorem~\ref{thm_main2}) and a weaker result can be found in \cite{D14}. %The main idea of the proof of this result, which will be given in Section~\ref{sec_proof_main}, is to combine the approach of \cite{DP14a} (see also \cite{DP14b} for an informal argument) with the approach of \cite{M15} and \cite{BM15}. 

%The explicit construction will be described in Section~\ref{sec2}. From there it is clear that the construction does not depend on $p$. In other words, we have a fixed construction which yields the optimal result for all $1 < p < \infty$. \\

Proinov's lower bound \eqref{lbdlpdiseq} and Theorem~\ref{main_result} give the precise behavior of the $L_p$-discrepancy for $1 < p < \infty$. On the other hand, the $L_\infty$-discrepancy remains elusive. We have constructions of infinite sequences $\S_d$ in $[0,1)^d$ (for example order 1 digital $(t,d)$-sequences as presented in Section~\ref{sec2}, see \cite{DP10,N87,N92}) such that $$L_{\infty,N}(\S_d) \ll_d \frac{(\log N)^d}{N}.$$ As to lower bounds, we know that there exists some $c_d >0$ and $\eta_d\in (0,\tfrac{1}{2})$ such that for every sequence $\S_d$ in $[0,1)^d$ we have $$L_{\infty,N}(\S_d) \ge c_d \frac{(\log N)^{\frac{d}{2}+\eta_d}}{N} \ \ \ \mbox{for infinitely many $N \in \mathbb{N}$.}$$ This result follows from a corresponding result for finite point sets by Bilyk, Lacey and Vagharshakyan~\cite{BLV08}. For growing $d$ the exponent $\eta_d$ in this estimate tends to zero.

In dimension $d=1$ we even know that for every sequence $\S_1 \in [0,1)$ we have $$L_{\infty,N}(\S_1) \ge c \frac{\log N}{N} \ \ \ \mbox{for infinitely many $N \in \mathbb{N}$}$$ for some positive $c$. This is a famous result of Schmidt~\cite{S72} (see also \cite{B82,L14}). Since the $L_{\infty}$-discrepancy of the van der Corput sequence is of order $(\log N)/N$ the exact order of the $L_{\infty}$-discrepancy of infinite sequences in dimension $d=1$ is known. However, the quest for the exact order of the $L_{\infty}$-discrepancy in the multivariate case is a very demanding open question.\\

It is a natural question to ask what happens in intermediate spaces ``close'' to $L_{\infty}$. Standard examples of such spaces are $\bmo$-spaces and exponential Orlicz spaces. For point sets, the norm of the discrepancy function in these spaces was studied in \cite{BLPV09,BM15}. The methods of this paper can also be used to give sharp bounds for sequences. This will be the subject of a follow-up paper \cite{DHMP16}. 

Moreover, it is well-known that norms of the discrepancy function are intimately connected to integration errors of the corresponding quasi-Monte Carlo rules. That is the reason for recent work on the discrepancy function in function spaces like Sobolev spaces, Besov spaces and Triebel-Lizorkin spaces of dominating mixed smoothness, see \cite{H10,M13a,M13b,M13c,M15,T10,T10a}. Again the methods of our paper can be used to give sharp bounds for sequences. This will also be treated in \cite{DHMP16}. \\

The explicit construction in Theorem~\ref{main_result} is based on linear algebra over the finite field $\mathbb{F}_2$. In the subsequent section we provide a detailed introduction to the infinite sequences which lead to the optimal discrepancy bounds.

\section{Digital nets and sequences}\label{sec2}

\subsection{The digital construction scheme according to Niederreiter}

The concepts of digital nets and sequences were introduced by Niederreiter~\cite{N87} in 1987. These constructions are based on linear algebra over $\mathbb{F}_b$, the finite field of prime-power order $b$. A detailed overview of this topic is given in the books \cite{DP10,N92} (see also \cite[Chapter~5]{LP14}). Here we restrict ourselves to the case $b=2$. Let $\mathbb{F}_2$ be the finite field of order 2. We identify $\mathbb{F}_2$ with $\{0,1\}$ equipped with arithmetic operations modulo 2.

First we recall the definition of digital nets according to Niederreiter, which we present here in a slightly more general form. For $n,q,d \in \N$ with $q \ge n$ let $C_1,\ldots, C_d \in \mathbb{F}_2^{q \times n}$ be $q \times n$ matrices over $\mathbb{F}_2\cong \{0,1\}$ (originally one uses $n \times n$ matrices). For $k \in \{0,\ldots ,2^n-1\}$ with dyadic expansion $k = k_0 + k_1 2 + \cdots + k_{n-1} 2^{n-1}$, where $k_j \in \{0,1\}$, we define the dyadic digit vector $\vec{k}$ as $\vec{k} = (k_0, k_1, \ldots, k_{n-1})^\top \in \mathbb{F}_2^n$ (the symbol $\top$ means the transpose of a vector or a matrix; hence $\vec{k}$ is a column-vector). Then compute
\begin{equation}\label{matrix_vec_net}
C_j \vec{k} =:(x_{j,k,1}, x_{j,k,2},\ldots,x_{j,k,q})^\top \quad \mbox{for } j = 1,\ldots, d,
\end{equation}
where the matrix vector product is evaluated over $\mathbb{F}_2$, and put
\begin{equation*}
x_{j,k} = x_{j,k,1} 2^{-1} + x_{j,k,2} 2^{-2} + \cdots + x_{j,k,q} 2^{-q} \in [0,1).
\end{equation*}
The $k$-th point $\boldsymbol{x}_k$ of the net $\P_{2^n,d}$ is given by $\boldsymbol{x}_k = (x_{1,k}, \ldots, x_{d,k})$. A net $\P_{2^n,d}$ constructed this way is called a {\it digital net (over $\mathbb{F}_2$) with generating matrices} $C_1,\ldots,C_d$. Note that a digital net consists of $2^n$ elements in $[0,1)^d$.

A variant of digital nets are so-called {\it digitally shifted digital nets}. Here one chooses $(\vec{\sigma}_1,\ldots,\vec{\sigma}_d) \in (\mathbb{F}_2^{\mathbb{N}})^d$ with $\vec{\sigma}_j=(\sigma_{j,1},\sigma_{j,2},\ldots)^{\top} \in \mathbb{F}_2^{\mathbb{N}}$ with all but finitely many components different from zero and replaces \eqref{matrix_vec_net} by
 \begin{equation*}\label{matrix_vec_net2}
C_j \vec{k} +\vec{\sigma}_j =:(x_{j,k,1}, x_{j,k,2},x_{j,k,3},\ldots,)^\top \in \mathbb{F}_2^{\mathbb{N}} \quad \mbox{for } j = 1,\ldots, d,
\end{equation*}
and puts
\begin{equation*}
x_{j,k} = x_{j,k,1} 2^{-1} + x_{j,k,2} 2^{-2} + x_{j,k,3} 2^{-3} +\cdots  \in [0,1).
\end{equation*}

We also recall the definition of digital sequences according to Niederreiter, which are infinite versions of digital nets. Let $C_1,\ldots, C_d \in \mathbb{F}_2^{\mathbb{N} \times \mathbb{N}}$ be $\mathbb{N} \times \mathbb{N}$ matrices over $\mathbb{F}_2$. For $C_j = (c_{j,k,\ell})_{k, \ell \in \mathbb{N}}$ we assume that for each $\ell \in \mathbb{N}$ there exists a $K(\ell) \in \mathbb{N}$ such that $c_{j,k,\ell} = 0$ for all $k > K(\ell)$. For $k \in \N_0$ with dyadic expansion $k = k_0 + k_1 2 + \cdots + k_{m-1} 2^{m-1} \in \mathbb{N}_0$,  define the infinite dyadic digit vector of $k$ by $\vec{k} = (k_0, k_1, \ldots, k_{m-1}, 0, 0, \ldots )^\top \in \mathbb{F}_2^{\mathbb{N}}$. Then compute
\begin{equation}\label{eq_dig_seq}
C_j \vec{k}=:(x_{j,k,1}, x_{j,k,2},\ldots)^\top \quad \mbox{for } j = 1,\ldots, d,
\end{equation}
where the matrix vector product is evaluated over $\mathbb{F}_2$, and put
\begin{equation*}
x_{j,k} = x_{j,k,1} 2^{-1} + x_{j,k,2} 2^{-2} + \cdots \in [0,1).
\end{equation*}
The $k$-th point $\boldsymbol{x}_k$ of the sequence $\S_d$ is given by $\boldsymbol{x}_k = (x_{1,k}, \ldots, x_{d,k})$. A sequence $\S_d$ constructed this way is called a {\it digital sequence (over $\mathbb{F}_2$) with generating matrices} $C_1,\ldots,C_d$. Note that since $c_{j,k,\ell}=0$ for all $k$ large enough, the numbers $x_{j,k}$ are always dyadic rationals, i.e., have a finite dyadic expansion. 

The variant of {\it digitally shifted digital sequences} is defined in the same way as was done for digitally shifted digital nets.

\subsection{Higher order nets and sequences}\label{sec_honetssequ}

Our approach is based on higher order digital nets and sequences constructed explicitly in \cite{D07,D08}. We state here simplified versions of these definitions which are sufficient for our purpose. 

The distribution quality of digital nets and sequences depends on the choice of the respective generating matrices. In the following definitions we put some restrictions on $C_1,\ldots ,C_d$ with the aim to quantify the quality of equidistribution of the digital net or sequence.

\begin{df}\rm\label{def_net}
Let $n, q, \alpha \in \N$ with $q \ge \alpha n$ and let $t$ be an integer such that $0 \le t \le \alpha n$. Let $C_1,\ldots, C_d \in \mathbb{F}_2^{q \times n}$. Denote the $i$-th row vector of the matrix $C_j$ by $\vec{c}_{j,i} \in \mathbb{F}_2^n$. If for all $1 \le i_{j,\nu_j} < \ldots <
i_{j,1} \le q$ with $$\sum_{j = 1}^d \sum_{l=1}^{\min(\nu_j,\alpha)} i_{j,l}  \le
\alpha n - t$$ the vectors
$$\vec{c}_{1,i_{1,\nu_1}}, \ldots, \vec{c}_{1,i_{1,1}}, \ldots,
\vec{c}_{d,i_{d,\nu_d}}, \ldots, \vec{c}_{d,i_{d,1}}$$ are linearly independent
over $\mathbb{F}_2$, then the digital net with generating matrices
$C_1,\ldots, C_d$ is called an {\it order $\alpha$ digital $(t,n,d)$-net over $\mathbb{F}_2$}.
\end{df}

The case $\alpha=1$ corresponds to the classical case of $(t,n,d)$-nets according to Niederreiter's definition in  \cite{N87}.

Next we consider digital sequences for which the initial segments are order $\alpha$ digital $(t,n,d)$-nets over $\mathbb{F}_2$.

\begin{df}\rm\label{def_seq}
Let $\alpha \in \N$ and let $t \ge 0$ be an integer. Let $C_1,\ldots, C_d \in \mathbb{F}_2^{\mathbb{N} \times \mathbb{N}}$ and let $C_{j, \alpha n \times n}$ denote the left upper $\alpha n \times n$ submatrix of  $C_j$. If for all $n > t/\alpha$ the matrices $C_{1, \alpha n \times n},\ldots, C_{d, \alpha n \times n}$ generate an order $\alpha$ digital $(t,n,d)$-net over $\mathbb{F}_2$, then the digital sequence with generating matrices $C_1,\ldots, C_d$ is called an {\it order $\alpha$ digital $(t,d)$-sequence over $\mathbb{F}_2$}.
\end{df}

Again, the case $\alpha=1$ corresponds to the classical case of $(t,d)$-sequences according to Niederreiter's definition in \cite{N87}.

From Definition~\ref{def_net} it is clear that if $\P_{2^n,d}$ is an order $\alpha$ digital $(t,n,d)$-net, then for any $t \le t' \le \alpha n$, $\P_{2^n,d}$ is also an order $\alpha$ digital $(t',n,d)$-net. An analogous result also applies to higher order digital sequences.

Note that a digital net can be an order $\alpha$ digital $(t,n,d)$-net over $\mathbb{F}_2$ and at the same time an order $\alpha'$ digital $(t',n,d)$-net over $\mathbb{F}_2$ for $\alpha'\not = \alpha$. This means that the quality parameter $t$ may depend on $\alpha$ (i.e., $t=t(\alpha)$). The same holds for digital sequences. In particular \cite[Theorem~4.10]{D08} implies that an order $\alpha$ digital $(t,n,d)$-net is an order $\alpha'$ digital $(t',n,d)$-net for all $1 \le \alpha' \le \alpha$ with
\begin{equation}\label{eq_t_tprime}
t' = \lceil t \alpha'/\alpha \rceil \le t.
\end{equation}
The same result applies to order $\alpha$ digital $(t,d)$-sequences which are also order $\alpha'$ digital $(t',d)$-sequences with $1 \le \alpha' \le \alpha$ and $t'$ as above. In other words, $t(\alpha') = \lceil t(\alpha) \alpha'/\alpha \rceil$ for all $1 \le \alpha' \le \alpha$. More information can be found in \cite[Chapter~15]{DP10}.

In \cite{DP14a} it has been shown that every order $\alpha$ digital $(t,d)$-sequence over $\mathbb{F}_2$ with $\alpha \ge 5$ has optimal order of the $L_2$-discrepancy. In this paper we show that even order $2$ digital $(t,d)$-sequences over $\mathbb{F}_2$ achieve the optimal order of magnitude in $N$ of the $L_p$-discrepancy for all $p \in(1,\infty)$.  

Higher order digital nets and sequences have also a geometrical interpretation. Roughly speaking the definitions imply that special intervals or unions of intervals of prescribed volume contain the correct share of points with respect to a perfect uniform distribution. See \cite{N92,DP10} for the classical case $\alpha=1$ and \cite{DB09} or \cite{DP10,DP14a} for general $\alpha$. See also Lemma~\ref{fairint} below.

\subsection{Explicit constructions of order 2 digital sequences}\label{sec_exp_constr}

Explicit constructions of order $\alpha$ digital nets and sequences have been provided by Dick~\cite{D07,D08}. For our purposes it suffices to consider only $\alpha=2$. 

Let $C_1, \ldots, C_{2 d}$ be generating matrices of a digital net or sequence and let $\vec{c}_{j,k}$ denote the $k$-th row of $C_j$. Define matrices $E_1,\ldots, E_d$, where the $k$-th row of $E_j$ is given by $\vec{e}_{j,k}$, in the following way. For all $j \in \{1,2,\ldots,d\}$, $u \in \N_0$ and $v \in \{1,2\}$ let
\begin{equation*}
\vec{e}_{j,2 u + v} = \vec{c}_{2 (j-1) + v, u+1}.
\end{equation*}
We illustrate the construction for $d=1$. Then 
$$C_1=\left(\begin{array}{c} 
             \vec{c}_{1,1}\\
             \vec{c}_{1,2}\\
             \vdots 
            \end{array}\right), \ C_2=\left(\begin{array}{c} 
             \vec{c}_{2,1}\\
             \vec{c}_{2,2}\\
             \vdots 
            \end{array}\right) \ \Rightarrow \ 
           E_1=\left(\begin{array}{c} 
             \vec{c}_{1,1}\\
             \vec{c}_{2,1}\\
             \vec{c}_{1,2}\\
             \vec{c}_{2,2}\\
             \vdots 
            \end{array}\right).$$
This procedure is called \emph{interlacing} (in this case the so-called interlacing factor is $2$).

Recall that above we assumed that $c_{j,k,\ell}=0$ for all $k > K(\ell)$. Let $E_j = (e_{j,k,\ell})_{k, \ell \in \mathbb{N}}$. Then the construction yields that $e_{j,k,\ell} = 0$ for all $k > 2 K(\ell)$. 

From \cite[Theorem~4.11 and Theorem~4.12]{D07} we obtain the following result.
\begin{prp}
If $C_1,\ldots,C_{2d} \in\mathbb{F}_2^{\N \times \N}$ generate an order 1 digital $(t',2d)$-sequence over $\mathbb{F}_2$, then $E_1,\ldots,E_d \in\mathbb{F}_2^{\N \times \N}$ generate an order 2 digital $(t,d)$-sequence over $\mathbb{F}_2$ with $$t = 2 t' + d.$$
\end{prp}

Explicit constructions of suitable generating matrices $C_1,\ldots, C_s$ over $\mathbb{F}_2$ were obtained by Sobol'~\cite{S67}, Niederreiter~\cite{N87,N92}, Niederreiter-Xing~\cite{NX96} and others (see \cite[Chapter~8]{DP10} for an overview). Any of these constructions is sufficient for our purpose, however, for completeness, we briefly describe a special case of Tezuka's construction~\cite{T93}, which is a generalization of Sobol's construction~\cite{S67} and Niederreiter's construction~\cite{N87} of the generating matrices. 

We explain how to construct the entries $c_{j,k,\ell} \in \mathbb{F}_2$ of the generating matrices $C_j = (c_{j,k,\ell})_{k,\ell \ge 1}$ for $j=1,2,\ldots,s$ (for our purpose $s=2d$). To this end choose the polynomials $p_1=x$ and $p_j \in \mathbb{F}_2[x]$ for $j =2,\ldots,s$ to be the $(j-1)$-th irreducible polynomial in a list of irreducible polynomials over $\mathbb{F}_2$ that is sorted in increasing order according to their degree $e_j = \deg(p_j)$, that is, $e_2 \le e_3 \le \ldots \le e_{s-1}$ (the ordering of polynomials with the same degree is irrelevant). We also put $e_1=\deg(x)=1$. %(We point out that Niederreiter~\cite{N87} uses irreducible polynomials instead of primitive polynomials.)

Let $j \in \{1,\ldots,s\}$ and $k \in \mathbb{N}$. Take $i-1$ and $z$ to be respectively the main term and remainder when we divide $k-1$ by $e_j$, so that   $k-1  = (i-1) e_j + z$, with $0 \le z < e_j$. Now consider the Laurent series expansion
\begin{equation*}
\frac{x^{e_j-z-1}}{p_j(x)^i} = \sum_{\ell =1}^\infty a_\ell(i,j,z) x^{-\ell} \in \mathbb{F}_2((x^{-1})).
\end{equation*}
For $\ell \in \mathbb{N}$ we set
\begin{equation}\label{def_sob_mat}
c_{j,k,\ell} = a_\ell(i,j,z).
\end{equation}
Every digital sequence with generating matrices $C_j = (c_{j,k,\ell})_{k,\ell \ge 1}$ for $j=1,2,\ldots,s$ found in this way is a special instance of a Sobol' sequence which in turn is a special instance of so-called generalized Niederreiter sequences (see \cite[(3)]{T93}). Note that in the construction above we always have $c_{j,k,\ell}=0$ for all $k > \ell$. The $t$-value for these sequences is known to be $t = \sum_{j=1}^s (e_j-1)$, see \cite[Chapter~4.5]{N92} for the case of Niederreiter sequences.

\begin{rem}
Let $C_1,\ldots,C_{2d}$ be $\mathbb{N}\times \mathbb{N}$ matrices which are constructed according to Tezuka's method as described above. Let $E_1,\ldots,E_d$ be the generator matrices of the corresponding order 2 digital sequence. Then we always have $e_{j,k,\ell}=0$ for all $k > 2 \ell$, where $e_{j,k,\ell}$ is the entry in row $k$ and column $\ell$ of the matrix $E_j$.  
\end{rem}

Now we can state a more concrete version of our main result.

\begin{thm}\label{thm_main2}
For every order 2 digital $(t,d)$-sequence $\S_d$ over $\mathbb{F}_2$, with generating matrices $E_i=(e_{i,k,\ell})_{k,\ell \ge 1}$ for which $e_{i,k,\ell}=0$ for all $k > 2\ell$ and for all $i\in \{1,2,\ldots,d\}$,  we have  
\begin{align*}
 L_{p,N}(\S_d) \ll_{p,d} 2^t \frac{\left(\log N\right)^{\frac{d}{2}}}{N}\ \ \ \ \mbox{ for all $N\ge 2$ and all $1 \le p < \infty$.}
\end{align*}
\end{thm}

We remark that this result is not only a generalization of the main result in \cite{DP14a} from $L_2$- to $L_p$-discrepancy for general finite $p$ but also a considerable improvement in the following sense. In \cite{DP14a} the explicit construction is based on higher order sequences of order $\alpha=5$. Here, on the other hand, we show that even $\alpha=2$ suffices in order to achieve the optimal discrepancy bound with respect to the lower bound \eqref{lbdlpdiseq}. This means that for the explicit construction of a sequence in dimension $d$ one can begin with a classical digital sequence in dimension $s=2d$ rather than $s=5d$.

Note that the explicit construction of the order $2$ digital $(t,d)$-sequences $\S_d$ shown above, with generating matrices $E_i=(e_{i,k,\ell})_{k,\ell \ge 1}$ for which $e_{i,k,\ell}=0$ for all $k > 2\ell$ and for all $i\in \{1,2,\ldots,d\}$, does not depend on the parameter $p$ in Theorem~\ref{thm_main2}. Our explicit construction (based on Tezuka's construction and the interlacing of the generating matrices) is also extensible in the dimension, i.e., if we have constructed the sequence $\S_d$, we can add one more coordinate to obtain the sequence $\S_{d+1}$. In other words, we can define a sequence $\S_{\infty}$ of points in $[0,1)^{\mathbb{N}}$ and obtain the sequence $\S_d$, $d \in \mathbb{N}$, by projecting $\S_{\infty}$ to the first $d$ coordinates.

\section{Haar bases}\label{sec_Haar}

The proof of Theorem~\ref{thm_main2} is based on Haar functions. This is in contrast to the proof of the result in \cite{DP14a} (the $L_2$-discrepancy of order 5 digital sequences is of optimal order) which is based on Walsh functions.\\

We define $\N_0=\N \cup \{0\}$ and $\N_{-1}=\N_0\cup\{-1\}$. Let $\D_j = \{0,1,\ldots, 2^j-1\}$ for $j \in \N_0$ and $\D_{-1} = \{0\}$. For $\bsj = (j_1,\dots,j_d)\in\N_{-1}^d$ let $\D_{\bsj} = \D_{j_1}\times\ldots\times \D_{j_d}$. For $\bsj\in\N_{-1}^d$ we write $|\bsj| = \max(j_1,0) + \cdots + \max(j_d,0)$.

For $j \in \N_0$ and $m \in \D_j$ we call the interval $I_{j,m} = \big[ 2^{-j} m, 2^{-j} (m+1) \big)$ the $m$-th dyadic interval in $[0,1)$ on level $j$. We put $I_{-1,0}=[0,1)$ and call it the $0$-th dyadic interval in $[0,1)$ on level $-1$. Let $I_{j,m}^+ = I_{j + 1,2m}$ and $I_{j,m}^- = I_{j + 1,2m+1}$ be the left and right half of $I_{j,m}$, respectively. For $\bsj \in \N_{-1}^d$ and $\bsm = (m_1, \ldots, m_d) \in \D_{\bsj}$ we call $I_{\bsj,\bsm} = I_{j_1,m_1} \times \ldots \times I_{j_d,m_d}$ the $\bsm$-th dyadic interval in $[0,1)^d$ on level $\bsj$. We call the number $|\bsj|$ the order of the dyadic interval $I_{\bsj,\bsm}$. Its volume is $2^{-|\bsj|}$.

Let $j \in \N_{0}$ and $m \in \D_j$. Let $h_{j,m}$ be the function on $[0,1)$ with support in $I_{j,m}$ and the constant values $1$ on $I_{j,m}^+$ and $-1$ on $I_{j,m}^-$. We put $h_{-1,0} = \chi_{I_{-1,0}}$ on $[0,1)$. The function $h_{j,m}$ is called the {\it $m$-th dyadic Haar function on level $j$}.

Let $\bsj \in \N_{-1}^d$ and $\bsm \in \D_{\bsj}$. The function $h_{\bsj,\bsm}$ given as the tensor product
\[ h_{\bsj,\bsm}(\bsx) = h_{j_1,m_1}(x_1) \cdots h_{j_d,m_d}(x_d) \]
for $\bsx = (x_1, \ldots, x_d) \in [0,1)^d$ is called a {\it dyadic Haar function} on $[0,1)^d$. The system of dyadic Haar functions $h_{\bsj,\bsm}$ for $\bsj \in \N_{-1}^d, \, \bsm \in \D_{\bsj}$ is called {\it dyadic Haar basis on $[0,1)^d$}.

It is well known that the system
\[ \left\{2^{\frac{|\bsj|}{2}}h_{\bsj,\bsm} \,:\,\bsj\in\N_{-1}^d,\,\bsm\in \D_{\bsj}\right\} \]
is an orthonormal basis of $L_2([0,1)^d)$, an unconditional basis of $L_p([0,1)^d)$ for $1 < p < \infty$ and a conditional basis of $L_1([0,1)^d)$. For any function $f\in L_2([0,1)^d)$ we have Parseval's identity
\[ \|f|L_2([0,1)^d)\|^2 = \sum_{\bsj \in \N_{-1}^d} 2^{|\bsj|} \sum_{\bsm\in \D_{\bsj}}|\langle f,h_{\bsj,\bsm}\rangle|^2, \] where $\langle \cdot , \cdot \rangle$ denotes the usual $L_2$-inner product, i.e., $\langle f ,g \rangle =\int_{[0,1]^d} f(\bsx) g(\bsx) \dint \bsx$. The terms $\langle f,h_{\bsj,\bsm}\rangle$ are called the {\it Haar coefficients} of the function $f$.

The following Littlewood-Paley type estimate for the $L_p$-norm for $p \in (1,\infty)$ is a special case of \cite[Theorem~2.11, Corollary~1.13]{M13b}.
\begin{prp} \label{haarlpnorm}
Let $p \in (1,\infty)$ and $f\in L_p([0,1)^d)$. Then
\[ \|f|L_p([0,1)^d)\|^2 \ll_{p,d} \sum_{\bsj\in\N_{-1}^d} 2^{2|\bsj|(1-1/\bar{p})}\left(\sum_{\bsm\in\D_{\bsj}} |\langle f,h_{\bsj,\bsm}\rangle|^{\bar{p}}\right)^{2/\bar{p}} \]
where $\bar{p} = \max(p,2)$.
\end{prp}

We present a connection between higher order digital nets over $\mathbb{F}_2$ and dyadic intervals.

\begin{lem}\label{fairint}
Let $\P_{2^n,d}$ be an order $\alpha$ digital $(t,n,d)$-net over $\mathbb{F}_2$. Then every dyadic interval of order $n-\lceil t/\alpha\rceil$ contains at most $2^{\lceil t/\alpha\rceil}$ points of $\P_{2^n,d}$.   
\end{lem}

\begin{proof}
As mentioned in Section~\ref{sec_honetssequ}, every order $\alpha$ digital $(t,n,d)$-net over $\mathbb{F}_2$ is an order~1 digital $(\lceil t/\alpha \rceil,n,d)$-net over $\mathbb{F}_2$. Then every dyadic interval of order $n-\lceil t/\alpha\rceil$ contains exactly $2^{\lceil t/\alpha\rceil}$ points of $\P_{2^n,d}$ (see \cite{DP10,N92}). 
\end{proof}

The following lemma is a slight generalization of \cite[Lemma~5.9]{M15}. The result was originally proved for order $2$ digital $(t,n,d)$-nets. The extension to digitally shifted order $2$ digital $(t,n,d)$-nets follows with almost exactly the same arguments as the proof of \cite[Lemma~5.9]{M15} (not repeated here). We restrict ourselves to the finite field $\mathbb{F}_2$.

\begin{lem} \label{haarcoeffdignets}
 Let $\P_{2^n,d}$ be a digitally shifted order $2$ digital $(t,n,d)$-net over $\mathbb{F}_2$. Let $\bsj\in\N_{-1}^d$ with $|\bsj|+t/2\leq n$ and $\bsm\in\D_{\bsj}$. Then
\[ |\langle D_{\P_{2^n,d}},h_{\bsj,\bsm}\rangle| \ll 2^{-2n+t}(2n-t-2|\bsj|)^{d-1}. \]
\end{lem}

\section{The proof of Theorem~\ref{thm_main2}}\label{sec_proof_main}

For the proof of the main result we need some auxiliary lemmas. The first two lemmas are elementary:

\begin{lem}\label{index_dim_red}
Let $r\in\N_0$ and $s\in\N$. Then
\[ \#\{(a_1,\ldots,a_s)\in\N_0^s:\; a_1 + \cdots + a_s = r\} \leq (r + 1)^{s-1}. \]
\end{lem}
For a proof of this result we refer to \cite[Proof of Lemma~16.26]{DP10}.

\begin{lem}\label{index_dim_red_log}
 Let $K\in\N$, $A>1$ and $q,s \ge 0$. Then we have
\[ \sum_{r = 0}^{K-1} A^r (K-r)^q r^s \ll A^K\,K^s, \]
 where the implicit constant is independent of $K$.
\end{lem}
For a proof we refer to \cite[Lemma~5.2]{M15}.

The subsequent two lemmas are required in order to estimate the Haar coefficients of the discrepancy function. The first one is a special case of \cite[Lemma~5.1]{M13b}.

\begin{lem}\label{lem_haar_coeff_vol}
Let $f(\bsx) = x_1\cdots x_d$ for $\bsx=(x_1,\ldots,x_d)\in [0,1)^d$. Let $\bsj\in\N_{-1}^d$ and $\bsm\in\D_{\bsj}$. Then $|\langle f,h_{\bsj,\bsm}\rangle|\asymp 2^{-2|\bsj|}$.
\end{lem}

The next lemma is a special case of \cite[Lemma~5.2]{M13b}.
\begin{lem} \label{lem_haar_coeff_number}
Let $\bsz = (z_1,\ldots,z_d) \in [0,1)^d$ and $g(\bsx) = \chi_{[\bszero,\bsx)}(\bsz)$ for $\bsx = (x_1, \ldots, x_d) \in [0,1)^d$. Let $\bsj\in\N_{-1}^d$ and $\bsm\in\D_{\bsj}$. Then $\langle g,h_{\bsj,\bsm}\rangle = 0$ if $\bsz$ is not contained in the interior of the dyadic interval $I_{\bsj,\bsm}$. If $\bsz$ is contained in the interior of $I_{\bsj,\bsm}$ then $|\langle g,h_{\bsj,\bsm}\rangle|\ll 2^{-|\bsj|}$.
\end{lem}

Now we have collected all auxiliary results in order to provide the proof of Theorem~\ref{thm_main2}.

\begin{proof}[Proof of Theorem~\ref{thm_main2}]
According to the monotonicity of the $L_p$-norm it suffices to prove the result for $p>1$. Let $\S_d$ be an order $2$ digital $(t,d)$-sequence over $\mathbb{F}_2$ with generating matrices $C_1,\ldots,C_d$, with $C_j = (c_{j,k,\ell})_{k, \ell \ge 1}$ for which $c_{i,k,\ell} = 0$ for all $k > 2 \ell$. Let $N \in \N$ with dyadic expansion $N=2^{n_r}+\cdots+2^{n_1}$ with $n_r>\ldots>n_1\geq 0$. 

We first prove the following claim: For $\mu = 1, \ldots, r$ let
\begin{equation*}
\Q_{2^{n_\mu},d} =\{\boldsymbol{x}_{2^{n_1}+\cdots + 2^{n_{\mu-1}}}, \boldsymbol{x}_{2^{n_1}+\cdots + 2^{n_{\mu-1}} + 1}, \ldots, \boldsymbol{x}_{-1+2^{n_1} + \cdots + 2^{n_\mu}}\},
\end{equation*}
where for $\mu = 1$ we set $2^{n_1} + \ldots + 2^{n_{\mu-1}} = 0$. Then the point set $\P_{N,d}$ consisting of the first $N$ elements of the sequence $\S_d$ is a union of $\Q_{2^{n_\mu}, d}$ for $\mu = 1, \ldots, r$ and $\Q_{2^{n_\mu}, d}$ is a digitally shifted order $2$ digital $(t,n_\mu,d)$-net over $\mathbb{F}_2$ with generating matrices $C_{1,2 n_\mu \times n_\mu},\ldots,C_{d,2 n_\mu \times n_\mu}$, i.e., the left upper $2 n_\mu \times n_\mu$ submatrices of $C_1,\ldots,C_d$. 

For the proof of this claim let $C_{j, \mathbb{N} \times n_\mu}$ denote matrix which consists of the first $n_\mu$ columns of $C_j$. Only the first $2 n_\mu$ rows of $C_{j, \mathbb{N} \times n_\mu}$ can be nonzero since $c_{j,k,\ell} = 0$ for all $k > 2 \ell$ and hence $C_j$ is of the form
$$
C_{j} = \left( \begin{array}{ccc}
           & \vline &  \\
  C_{j,2 n_\mu \times n_\mu} & \vline & D_{j,2 n_\mu \times \mathbb{N}} \\
           & \vline &   \\ \hline
    & \vline & \\
 0_{\mathbb{N} \times n_\mu}  & \vline &  F_{j,\mathbb{N} \times \mathbb{N}} \\
   &   \vline       &
\end{array} \right) \in \mathbb{F}_2^{\mathbb{N} \times \mathbb{N}},
$$  where $0_{\mathbb{N} \times n_\mu}$ denotes the $\mathbb{N} \times n_\mu$ zero matrix. Note that the entries of each column of the   matrix $F_{j,\mathbb{N} \times \mathbb{N}}$ become eventually zero. Any $k \in \{2^{n_1}+\cdots + 2^{n_{\mu -1}}, 2^{n_1}+\cdots + 2^{n_{\mu-1}} + 1, \ldots ,-1+2^{n_1} + \cdots + 2^{n_\mu}\}$ can be written in the form $$k=2 ^{n_1}+\cdots+2^{n_{\mu-1}}+a=2^{n_{\mu-1}} \ell +a$$ with $a \in \{0,1,\ldots,2^{n_\mu}-1\}$ and $\ell=1+2^{n_{\mu}-n_{\mu-1}}+\cdots+2^{n_1-n_{\mu-1}}$ if $\mu > 1$ and $\ell =0$ if $\mu=1$. Hence the dyadic digit vector of $k$ is given by $$\vec{k}=(a_0,a_1,\ldots,a_{n_\mu-1},l_0,l_1,l_2,\ldots)^{\top}=:{\vec{a} \choose \vec{\ell}},$$ where $a_0,\ldots,a_{n_\mu-1}$ are the dyadic digits of $a$ and $l_0,l_1,l_2,\ldots$ are the dyadic digits of $\ell$. With this notation we have
$$C_j \vec{k}=\left( \begin{array}{c} C_{j,2 n_\mu \times n_\mu}  \vec{a} \\ 0 \\ 0 \\ \vdots \end{array} \right)  + \left( \begin{array}{c}
             \\
  D_{j,2 n_\mu \times \mathbb{N}}  \\
          \\ \hline  \\
 F_{j,\mathbb{N} \times \mathbb{N}}  \\
   \end{array} \right) \vec{\ell}.$$ For the point set $\Q_{2^{n_\mu},d}$ under consideration, the vector
\begin{equation}\label{dig_shift}
\vec{\sigma}_{\mu,j}:=\left( \begin{array}{c}
             \\
  D_{j,\alpha n_\mu \times \mathbb{N}}  \\
          \\ \hline  \\
 F_{j,\mathbb{N} \times \mathbb{N}}  \\
   \end{array} \right) \vec{\ell}
\end{equation}
is constant and its components become eventually zero (i.e., only a finite number of components is nonzero). Furthermore, $C_{j,2 n_\mu \times n_\mu}  \vec{a}$ for $a=0,1,\ldots,2^{n_\mu}-1$ and $j=1,\ldots,s$ generate an order 2 digital $(t,n_\mu,s)$-net over $\mathbb{F}_2$ (which is also an order 1 digital $(t,n_\mu,d)$-net over $\mathbb{F}_2$, which follows from \cite[Proposition~1]{D09}).

This means that the point set $\Q_{2^{n_\mu},d}$ is a digitally shifted order 2 digital $(t,n_\mu,d)$-net over $\mathbb{F}_2$ with generating matrices $C_{1, 2 n_\mu \times n_\mu},\ldots, C_{d, 2 n_\mu \times n_\mu}$ and hence the claim is proven.

%Hence $\P_{N,d} =\bigcup_{\kappa=1}^r \Q_{2^{n_\kappa},d}$.

According to Proposition~\ref{haarlpnorm} with $f = D^{N}_{\S_d}$ we have 
\begin{equation}\label{bdlpdisc1}
(L_{p,N}(\mathcal{S}_d))^2 \ll_{p,d} \sum_{\bsj\in\N_{-1}^d} 2^{2|\bsj|(1-1/\bar{p})}\left(\sum_{\bsm\in\D_{\bsj}} |\langle D_{\mathcal{S}_d}^N,h_{\bsj,\bsm}\rangle|^{\bar{p}}\right)^{2/\bar{p}},
\end{equation}
where $\bar{p}=\max(p,2)$.

Now we split up the sum in \eqref{bdlpdisc1} according to the size $|\bsj|$ of the involved dyadic intervals. Let us consider ``small'' dyadic intervals first, so let $|\bsj|+t/2\geq \ld N$. We calculate $|\langle D_{\S_d}^N,h_{\bsj,\bsm}\rangle|$. Choose $n \in \mathbb{N}$ such that $2^{n-1} < N \le 2^n$. Then the point set $\P_{N,d}$ of the first $N$ elements of $\S_d$ is a subset of $\P_{2^n,d}$ consisting of the first $2^n$ elements of $\S_d$. From the construction of $\S_d$ it follows that $\P_{2^n,d}$ is an order $2$ digital $(t,n,d)$-net over $\mathbb{F}_2$. Therefore, according to Lemma~\ref{fairint}, in an interval $I_{\bsj,\bsm}$ there are at most $2^{\lceil t/2\rceil}$ points of $\P_{2^n,d}$ and hence  there are at most $2^{\lceil t/2\rceil}$ points of $\P_{N,d}$ in $I_{\bsj,\bsm}$. Hence we get from Lemmas~\ref{lem_haar_coeff_vol} and \ref{lem_haar_coeff_number}
\begin{align} \label{haarcoeffwithpoints}
 |\langle D_{\S_d}^N,h_{\bsj,\bsm}\rangle| \le \frac{1}{N} \frac{2^{\lceil t/2 \rceil}}{2^{|\bsj|}}+\frac{1}{2^{2|\bsj|}}  \ll \frac{1}{N} \frac{2^{t/2}}{2^{|\bsj|}}.
\end{align}
The estimate \eqref{haarcoeffwithpoints} will be applied to dyadic intervals on level $\bsj$ which contain points from $\P_{N,d}$. The cardinality of such intervals is at most $N$. At least $2^{|\bsj|}-N$ contain no points of $\P_{N,d}$, hence in such cases we get from Lemma~\ref{lem_haar_coeff_vol} and \ref{lem_haar_coeff_number}
\begin{align} \label{haarcoeffwithoutpoints}
 |\langle D_{\S_d}^N,h_{\bsj,\bsm}\rangle| \ll \frac{1}{2^{2|\bsj|}}.
\end{align}

Now we estimate the terms of \eqref{bdlpdisc1} for which $|\bsj|+t/2 \ge \ld N$. Applying Minkowski's inequality and inserting \eqref{haarcoeffwithpoints} and \eqref{haarcoeffwithoutpoints} we obtain 
\begin{eqnarray}\label{bd_small_int}
\lefteqn{\sum_{|\bsj|+t/2 \geq \ld N} 2^{2|\bsj|(1-1/\bar{p})} \left(\sum_{\bsm\in \D_{\bsj}}|\langle D_{\S_d}^N,h_{\bsj,\bsm}\rangle|^{\bar{p}}\right)^{2/\bar{p}}}\nonumber\\
&\ll & \sum_{|\bsj|+t/2 \geq \ld N} 2^{2|\bsj|(1-1/\bar{p})} \left(N \frac{1}{N^{\bar{p}}} \frac{2^{\bar{p}t/2}}{2^{\bar{p}|\bsj|}}\right)^{2/\bar{p}} +\sum_{|\bsj|+t/2 \geq \ld N} 2^{2|\bsj|(1-1/\bar{p})} \left(\frac{(2^{|\bsj|}-N)}{2^{2\bar{p}|\bsj|}}\right)^{2/\bar{p}}\nonumber \\
&\le& N^{2/\bar{p}-2} 2^{t} \sum_{|\bsj|+t/2 \geq \ld N} \frac{1}{2^{2|\bsj|/\bar{p}}} + \sum_{|\bsj|+t/2 \geq \ld N} \frac{1}{2^{2|\bsj|}}\nonumber \\
&\ll& N^{2/\bar{p}-2} 2^{t} N^{-2/\bar{p}} 2^{t/\bar{p}} (\ld N)^{d-1} + \frac{(\log N)^{d-1}}{N^2} 2^t\nonumber \\
&\le& 2^{2t} \frac{(\log N)^{d-1}}{N^2}.
\end{eqnarray}

We now turn to the more demanding case of ``large'' intervals where $|\bsj|+t/2<\ld N$. More precisely assume that we have
\[ n_{\mu} \le |\bsj|+t/2 < n_{\mu+1} \]
for some $\mu\in \{0, 1,\ldots,r\}$, where we set $n_0 = 0$ and $n_{r+1} = \ld N $. For $\kappa \in \{\mu+1,\ldots , r\}$ we use the estimation from Lemma~\ref{haarcoeffdignets} to get
\[ |\langle D_{\Q_{2^{n_\kappa},d}},h_{\bsj,\bsm}\rangle| \ll  \frac{(2n_\kappa - t - 2|\bsj|)^{d-1}}{2^{2n_\kappa - t}}, \]
while for $M = 2^{n_\mu}+\cdots+2^{n_1}$, with $\mu > 0$, we have according to \eqref{haarcoeffwithpoints} with $\widetilde{\Q}_{M,d} = \bigcup_{i=1}^{\mu} \Q_{2^{n_i}, d}$ that
\[ |\langle D_{\widetilde{\Q}_{M,d}},h_{\bsj,\bsm}\rangle| \ll \frac{1}{M} \frac{1}{2^{|\bsj|-t/2}}, \]
since $|\bsj|+t/2\ge\ld M$. If $\mu = 0$, this case does not occur.

Using the linearity of the discrepancy function and the triangle inequality leads to
\begin{align}\label{asdf}
 |\langle D_{\S_d}^N,h_{\bsj,\bsm}\rangle| &\le \frac{M}{N} |\langle D_{\widetilde{\Q}_{M,d} },h_{\bsj,\bsm}\rangle| + \frac{1}{N}\sum_{\kappa=\mu+1}^r 2^{n_\kappa} |\langle D_{\Q_{2^{n_\kappa}, d} },h_{\bsj,\bsm}\rangle|\nonumber\\
 &\ll \frac{1}{N}\left( \frac{2^{t/2}}{2^{|\bsj|}} + \sum_{\kappa=\mu+1}^r  \frac{(2n_\kappa - t - 2|\bsj|)^{d-1}}{2^{n_\kappa-t}}\right)\nonumber\\
 &\ll \frac{2^t}{N} \left(\frac{1}{2^{|\bsj|}} + \sum_{k=0}^{\infty}  \frac{(2n_{\mu+1} +2k - t - 2|\bsj|)^{d-1}}{2^{n_{\mu+1}+k}}\right)\nonumber\\
 &\ll \frac{2^t}{N} \left(\frac{1}{2^{|\bsj|}} +  \frac{(2n_{\mu+1} - t - 2|\bsj|)^{d-1}}{2^{n_{\mu+1}}}\right),
\end{align}
where we used \cite[Lemma~13.24]{DP10} in the last step.

Now we use this bound to estimate the terms of \eqref{bdlpdisc1} for which $|\bsj|+t/2 < \ld N$. We obtain from Minkowski's inequality, Lemma~\ref{index_dim_red} and Lemma~\ref{index_dim_red_log}
\begin{eqnarray}\label{bd_large_int} 
\lefteqn{\sum_{|\bsj|+t/2 < \ld N} 2^{2|\bsj|(1-1/\bar{p})} \left(\sum_{\bsm\in \D_{\bsj}}|\langle D_{\S_d}^N,h_{\bsj,\bsm}\rangle|^{\bar{p}}\right)^{2/\bar{p}}} \nonumber \\
&= & \sum_{\mu=0}^r \sum_{n_{\mu} \le |\bsj|+t/2 < n_{\mu+1}} 2^{2|\bsj|(1-1/\bar{p})} \left(\sum_{\bsm\in \D_{\bsj}}|\langle D_{S_d}^N,h_{\bsj,\bsm}\rangle|^{\bar{p}}\right)^{2/\bar{p}}\nonumber \\
&\ll &  \sum_{\mu=0}^r \sum_{n_{\mu} \le |\bsj|+t/2 < n_{\mu+1}} 2^{2|\bsj|(1-1/\bar{p})} 2^{2|\bsj|/\bar{p}} \, \frac{2^{2t}}{N^2} \left( \frac{1}{2^{|\bsj|}} +  \frac{(2n_{\mu+1} - t - 2|j|)^{d-1}}{2^{n_{\mu+1}}} \right)^2\nonumber \\
&= & \frac{2^{2t}}{N^2}  \sum_{\mu=0}^r \sum_{n_{\mu} \le |\bsj|+t/2 < n_{\mu+1}} 2^{2|\bsj|} \left( \frac{1}{2^{2|\bsj|}} + 2 \frac{(2n_{\mu+1} - t - 2|\bsj|)^{d-1}}{2^{|\bsj|+n_{\mu+1}}}\right.\nonumber \\
&& \left.\hspace{8cm} +  \frac{(2n_{\mu+1} - t - 2|\bsj|)^{2d-2}}{2^{2n_{\mu+1}}} \right)\nonumber \\
&= & \frac{2^{2t}}{N^2}  \sum_{\mu=0}^r \left( \sum_{n_{\mu} \le |\bsj|+t/2 < n_{\mu+1}} 1 + \frac{2}{2^{n_{\mu+1}}} \sum_{n_{\mu}< |\bsj|+t/2 \le n_{\mu+1}} 2^{|\bsj|}(2n_{\mu+1} - t - 2|\bsj|)^{d-1}\right.\nonumber \\ 
& &\hspace{3cm} \left. + \frac{1}{2^{2n_{\mu+1}}} \sum_{n_{\mu} \le |\bsj|+t/2 < n_{\mu+1}} 2^{2|\bsj|}(2n_{\mu+1} - t - 2|\bsj|)^{2d-2} \right)\nonumber \\
&\ll & \frac{2^{2t}}{N^2}  \sum_{\mu=0}^r \left((\log N)^{d-1} (n_{\mu+1}-n_{\mu}) + \frac{2^{n_{\mu+1}-t/2}}{2^{n_{\mu+1}}}  (\log N)^{d-1}+ \frac{2^{2n_{\mu+1}-t}}{2^{2n_{\mu+1}}}  (\log N)^{d-1} \right)\nonumber \\
&\le & \frac{2^{2t}}{N^2} \left( (\log N)^d + 2(\log N)^{d-1} (r+1) \right)\nonumber \\
&\ll & 2^{2t} \, \frac{(\log N)^d}{N^2}.
\end{eqnarray}
Combining \eqref{bd_small_int} and \eqref{bd_large_int} we obtain $$(L_{p,N}(\S_d))^2 \ll_{p,d} 2^{2t} \frac{(\log N)^{d}}{N^2}.$$ Now the result follows by taking square roots.
\end{proof}

\noindent {\sc Josef Dick} 

\noindent School of Mathematics and Statistics, The University of New South Wales, Sydney NSW 2052, Australia \\
email: josef.dick(AT)unsw.edu.au\\

\noindent {\sc Aicke Hinrichs}

\noindent Institut f\"ur Funktionalanalysis, Johannes Kepler Universit\"at Linz, Altenbergerstra{\ss}e 69, 4040 Linz, \"Osterreich\\
email: aicke.hinrichs(AT)jku.at\\

\noindent {\sc Lev Markhasin}

\noindent Institut f\"ur Stochastik and Anwendungen, Universit\"at Stuttgart, Pfaffenwaldring 57, 70569 Stuttgart, Deutschland\\
email: lev.markhasin(AT)mathematik.uni-stuttgart.de\\

\noindent {\sc Friedrich Pillichshammer}

\noindent Institut f\"ur Finanzmathematik und angewandte Zahlentheorie, Johannes Kepler Universit\"at Linz, Altenbergerstra{\ss}e 69, 4040 Linz, \"Osterreich\\
email: friedrich.pillichshammer(AT)jku.at

\end{document}